\documentclass[11pt]{amsart}
\usepackage{amsfonts}
\usepackage{amsmath}
\usepackage{amssymb}
\usepackage{color}

\usepackage[pagebackref,colorlinks=true,urlcolor=blue,linkcolor=blue,citecolor=blue]{hyperref}

\newcommand{\eps}{\epsilon}
\newcommand{\R}{\mathbb{R}}
\newcommand{\C}{\mathbb{C}}

\newcommand{\Q}{\mathbb{Q}}
\newcommand{\Z}{\mathbb{Z}}
\newcommand{\N}{\mathbb{N}}
\newcommand{\AK}{{\mathcal A}_K}
\newcommand{\AF}{{\mathcal A}_F}

\theoremstyle{plain}
  \newtheorem{theorem}[subsection]{Theorem}
  \newtheorem{proposition}[subsection]{Proposition}
  \newtheorem{lemma}[subsection]{Lemma}
  \newtheorem{corollary}[subsection]{Corollary}
  \newtheorem{question}[subsection]{Question}

\theoremstyle{remark}
  \newtheorem{remark}[subsection]{Remark}

\theoremstyle{definition}
  \newtheorem{definition}[subsection]{Definition}

\begin{document}
\title[$\{x+y,xy\}$ patterns]{Ergodic Theorem involving additive and multiplicative groups of a field and $\{x+y,xy\}$ patterns}
\author{Vitaly Bergelson}
\address[Bergelson]{Department of Mathematics, The Ohio State University, Columbus, OH 43210, USA}
\email{vitaly@math.ohio-state.edu}
\author{Joel Moreira}
\address[Moreira]{Department of Mathematics, The Ohio State University, Columbus, OH 43210, USA}
\email{moreira@math.ohio-state.edu}
\subjclass{Primary 37A45 \and 05D10\and Secondary 05A18\and 11T99\and 28D15}

\vspace{-0.3in}
\begin{abstract}
We establish a ``diagonal'' ergodic theorem involving the additive and multiplicative groups of a countable field $K$ and, with the help of a new variant of Furstenberg's correspondence principle, prove that any ``large'' set in $K$ contains many configurations of the form $\{x+y,xy\}$. We also show that for any finite coloring of $K$ there are many $x,y\in K$ such that $x,x+y$ and $xy$ have the same color.
Finally, by utilizing a finitistic version of our main ergodic theorem, we obtain combinatorial results pertaining to finite fields. In particular we obtain an alternative proof for a result obtained by Cilleruelo \cite{Cilleruelo12}, showing that for any finite field $F$ and any subsets $E_1,E_2\subset F$ with $|E_1||E_2|>6|F|$, there exist $u,v\in F$ such that $u+v\in E_1$ and $uv\in E_2$.

\end{abstract}
\thanks{The first author was supported by NSF grant DMS-1162073.}
\maketitle

\section{Introduction}
Schur's Theorem \cite{Schur16} asserts that, given a finite coloring of $\N=\{1,2,...\}$, there exist $x,y\in\N$ such that $x$, $y$ and $x+y$ all have the same color.
A multiplicative version of this theorem is also true: given a finite coloring of $\N$, we can find $x,y\in\N$ such that $x$, $y$ and $xy$ all have the same color. To see this, consider, for instance, the induced coloring of the set $\{2^n;n\in\N\}$.
However, very little is known regarding partition regularity of configurations involving both addition and multiplication.
For some results in this direction see \cite{Beiglbock_Bergelson_Hindman_Strauss06},  \cite[Section 6]{Bergelson10}, \cite{Hindman11}, \cite{Frantzikinakis_Host13}, \cite{Bergelson_McCutcheon96}, \cite{Bergelson05} and \cite{Bergelson_Furstenberg_McCutcheon96}.

In particular, the following question is still unanswered (cf. Question 3 in \cite{Hindman_Imre_Strauss03}).
\begin{question}\label{problem}
Given a finite coloring of $\N$, is it true that there exist distinct $x,y\in\N$ such that both $x+y$ and $xy$ have the same color? \footnote{In fact it is believed that one can find a monochromatic configuration of the form $\{x,y,x+y,xy\}$ for any finite coloring of $\N$.}
\end{question}

While replacing $\N$ by $\Z$ does not seem to make the question easier, replacing $\N$ by $\Q$ does and allows for the introduction of useful ergodic techniques.
In this paper we show that any `large'\footnote{Here large means to have positive upper density with respect to some \emph{double} F\o lner sequence in $\Q$. This will be defined in Section \ref{section_doublefolner}.} set in $\Q$ contains the sought-after configurations, which leads to a partition result involving three-element sets having the form $\{x,y+x,yx\}$.
Actually, the ergodic method that we employ works equally well in the framework of arbitrary countable\footnote{We use the word `countable' to mean infinitely countable} fields.\footnote{And indeed in the framework of finite fields, this is explored in Section \ref{section_finite}.}

Here is the formulation of the main partition result obtained in this paper.
\begin{theorem}[see Theorem \ref{thm_fullcolor} for a more precise formulation]\label{teo_main}
\

Let $K$ be a countable field. Given a finite coloring $K=\bigcup C_i$, there exists a color $C_i$ and $x,y\in K$ such that $\{x,x+y,xy\}\subset C_i$.
\end{theorem}
We remark that it follows from Theorem \ref{thm_fullcolor} below that $x$ and $y$ can in fact be choosen from outside any prescribed finite set.
We will derive Theorem \ref{teo_main} from a `density' statement which, in turn, follows from an ergodic result dealing with measure preserving actions of the affine group $\AK=\{ux+v:u,v\in K^*,u\neq0\}$ of the field $K$ (cf. Definition \ref{def_subgroups}). To formulate these results we need to introduce first the notion of double F\o lner sequences in $K$.

\begin{definition}\label{def_doublefolner}
Let $K$ be a countable field. A sequence of non-empty finite sets $(F_N)\subset K$ is called a \emph{double F\o lner sequence} if for each $x\in K^*:=K\setminus\{0\}$ we have
$$\lim_{N\to\infty}\frac{|F_N\cap(F_N+x)|}{|F_N|}=\lim_{N\to\infty}\frac{|F_N\cap(xF_N)|}{|F_N|}=1$$
\end{definition}
See Proposition \ref{prop_folner} for the proof of the existence of double F\o lner sequences in any countable field.

Here is the formulation of the main ergodic result in this paper and its combinatorial corollary (which is derived via a version of the Furstenberg correspondence principle - see Theorem \ref{prop_correspondence}).

\begin{theorem}\label{thm_introduction}
Let $K$ be a countable field and let $\AK$ be the group of affine transformations of $K$.
Let $(\Omega,{\mathcal B},\mu)$ be a probability space and let $(T_g)_{g\in\AK}$ be a measure preserving action of $\AK$ on $\Omega$.
For each $u\in K^*$, let $A_u=T_g$ where $g\in\AK$ is defined by $g:x\mapsto x+u$ and let $M_u=T_h$ where $h\in\AK$ is defined by $h:x\mapsto ux$.
Let $(F_N)$ be a double F\o lner sequence in $K$.
Then for each $B\in{\mathcal B}$ we have
$$\lim_{N\to\infty}\frac1{|F_N|}\sum_{u\in F_N}\mu(A_{-u}B\cap M_{1/u}B)\geq\mu(B)^2$$
and in particular the limit exists.
\end{theorem}

It is not hard to see that the quantity $\mu(B)^2$ in the right hand side of the displayed formula is the largest possible (consider for example the case when the action of $\AK$ is strongly mixing).

\begin{theorem}
Let $K$ be a countable field, let $(F_N)$ be a double F\o lner sequence in $K$ and let $E\subset K$ be such that $\bar d_{(F_N)}(E):=\limsup_{N\to\infty}|E\cap F_N|/|F_N|>0$.
Then there are infinitely many pairs $x,y\in K^*$ with $x\neq y$ such that
\begin{equation}\label{eq_main}
\{x+y,xy\}\subset E
\end{equation}
\end{theorem}

A precise formulation of how large is the set of pairs $(x,y)$ that satisfy equation (\ref{eq_main}) is given by Theorem \ref{teo_double} below.

We also obtain similar results for finite fields. For example, we have the following result (see also  Theorem \ref{teo_finitergodic} for a dynamical formulation)

\begin{theorem}\label{thm_introfinite}
For any finite field $F$ and any subsets $E_1,E_2\subset F$ with $\left|E_1\right|\left|E_2\right|>6|F|$, there exist $x,y\in F$, $y\neq0$, such that $x+y\in E_1$ and $xy\in E_2$.
\end{theorem}
Theorem \ref{thm_introfinite} has also been obtained for fields of prime order by Shkredov \cite{Shkredov10}, for general finite fields by Cilleruelo \cite[Corollary 4.2]{Cilleruelo12} and, with some additional quantitative estimates, by Hanson \cite{Hanson13}.

The paper is organized as follows.
In Section \ref{section_preliminaries} we introduce some notation, discuss basic facts about the affine group of a countable field, explore some properties of double F\o lner sequences, state our main results more precisely and obtain a general correspondence principle.
In Section \ref{section_proof} we prove our main ergodic theoretical results.
In Section \ref{section} we deduce the main partition result for infinite fields, Theorem \ref{teo_main}.
In Section \ref{section_finite} we adapt our methods to prove an analogue of Theorems \ref{teo_main} and \ref{thm_introduction} for finite fields.
Section \ref{section_remarks} is devoted to some general remarks.

\section{Preliminaries}\label{section_preliminaries}
\subsection{The group $\AK$ of affine transformations}
We will work with a fixed countable field $K$.
The set of non-zero elements of $K$ will be denoted by $K^*$.
It is not hard to see that for a set $E\subset K$, the statement that $E$ contains a configuration of the form $\{u+v,uv\}$ is equivalent to the statement that $(E-u)\cap(E/u)$ is non-empty for some $u\in K^*$.
To study this intersection we need to understand how the additive and the multiplicative groups of $K$ interact.
Hence it is natural to work with the subgroup of all bijections of $K$ generated by these two groups, which brings us to the group
of all affine transformations of $K$:

\begin{definition}\label{def_subgroups}
The \emph{affine group} of $K$ is the set
$$\AK=\{g:x\mapsto ux+v\mid u,v\in K, u\neq0\}$$
with the operation of composition of functions.
The \emph{additive subgroup} of $\AK$ is the set $S_A$ of affine transformations of the form $A_u:x\mapsto x+u$, with $u\in K$. Note that $S_A$ is isomorphic to the additive group $(K,+)$.
The \emph{multiplicative subgroup} of $\AK$ is the set $S_M$ of affine transformations of the form $M_u:x\mapsto ux$, with $u\in K^*$. Note that $S_M$ is isomorphic to the multiplicative group $(K^*,\times)$.
\end{definition}

 Note that the map $x\mapsto ux+v$ can be represented as the composition $A_vM_u$.
We have the following identity, which will be frequently utilized in this paper:
\begin{equation}\label{eq_affine}
M_uA_v=A_{uv}M_u
\end{equation}

Note that equation (\ref{eq_affine}) expresses the fact that $S_A$ is a normal subgroup of $\AK$.
Since both $S_A$ and $S_M\cong \AK/S_A$ are abelian groups, we conclude that $\AK$ is a solvable group.

We can now represent the intersection $(E-u)\cap(E/u)$ as $A_{-u}E\cap M_{1/u}E$,
where $A_{-u}E=\{x-u:x\in E\}$ and $M_{1/u}E=\{x/u:x\in E\}$.
With this notation, a set $E\subset K$ contains a configuration of the form $\{u+v,uv\}$ if and only if there is some $u\in K^*$ such that $A_{-u}E\cap M_{1/u}E$ is non-empty.
Also, the statement that there exist $u,v\in K^*$ such that $\{v,u+v,uv\}\subset E$ is equivalent to the statement that there exists some $u\in K^*$ such that $E\cap A_{-u}E\cap M_{1/u}E\neq\emptyset$.

\subsection{Double F\o lner sequences}\label{section_doublefolner}
The first step in the proof of Theorem \ref{teo_main} is to prove that the intersection $A_{-u}E\cap M_{1/u}E$ is non empty (for many choices of $u\in K$) when $E$ is a large subset of $K$ in a suitable sense.
In this section we will make this statement more precise.

As mentioned before, $\AK$ is solvable and hence it is a (discrete) countable amenable group.
This suggests the existence of a sequence of finite sets $(F_N)$ in $K$ asymptotically invariant under the action of $\AK$. Indeed we have the following:

\begin{proposition}\label{prop_folner}
Let $K$ be a countable field.
There exists a sequence of non-empty finite sets $(F_N)$ in $K$ which forms a F\o lner sequence for the actions of both the additive group $(K,+)$ and the multiplicative group $(K^*,\times)$. In other words, for each $u\in K^*$ we have:
$$\lim_{N\to\infty}\frac{|F_N\cap(F_N+u)|}{|F_N|}=\lim_{N\to\infty}\frac{|F_N\cap(uF_N)|}{|F_N|}=1$$
We call such a sequence $(F_N)$ a double F\o lner sequence.
\end{proposition}
\begin{proof}
Let $(G_N)_{N\in\N}$ be a (left) F\o lner sequence in $\AK$.
This means that $G_N$ is a non-empty finite subset of $\AK$ for each $N\in\N$, and that for each $g\in\AK$ we have
$$\lim_{N\to\infty}\frac{\left|G_N\cap(gG_N)\right|}{|G_N|}=1$$

Note that for $g_1,g_2\in\AK$, if $g_1\neq g_2$ then there is at most one solution $x\in K$ to the equation $g_1x=g_2x$.
Thus, for each $N\in\N$, we can find a point $x_N$ in the (infinite) field $K$ such that $g_ix_N\neq g_jx_N$ for all pairs $g_i,g_j\in G_N$ with $g_i\neq g_j$. It follows that $F_N:=\{gx_N:g\in G_N\}$ has $|G_N|$ elements.

Since $F_N\cap gF_N\supset\{hx_N:h\in G_N\cap gG_N\}$ we have $|F_N\cap gF_N|\geq|\{hx_N:h\in G_N\cap gG_N\}|=|G_N\cap gG_N|$.
Therefore:
$$1\geq\limsup_{N\to\infty}\frac{|F_N\cap gF_N|}{|F_N|}\geq\liminf_{N\to\infty}\frac{|F_N\cap gF_N|}{|F_N|}
\geq\lim_{N\to\infty}\frac{|G_N\cap gG_N|}{|G_N|}=1$$

Finally, putting $g=M_u$ and $g=A_u$ in the previous equation we get that $(F_N)$ is a F\o lner sequence for $(K^*,\times)$ and for $(K,+)$.
\end{proof}

From now on we fix a double F\o lner sequence $(F_N)$ in $K$.
For a set $E\subset K$, the lower density of $E$ with respect to $(F_N)$ is defined by the formula:
$$\underline d_{(F_N)}(E):=\liminf_{N\to\infty}\frac{|F_N\cap E|}{|F_N|}$$
and the upper density of $E$ with respect to $(F_N)$ is defined by the formula:
$$\bar d_{(F_N)}(E):=\limsup_{N\to\infty}\frac{|F_N\cap E|}{|F_N|}$$
Note that both the upper and lower densities $\underline d_{(F_N)}$ and $\bar d_{(F_N)}$ are invariant under affine transformations. In particular, for every $u\in K^*$ we have $\underline d_{(F_N)}(E/u)=\underline d_{(F_N)}(E-u)=\underline d_{(F_N)}(E)$ and $\bar d_{(F_N)}(E/u)=\bar d_{(F_N)}(E-u)=\bar d_{(F_N)}(E)$.

The following is the first step towards the proof of Theorem \ref{teo_main}:
\begin{theorem}\label{teo_double}
Let $E\subset K$ be such that $\bar d_{(F_N)}(E)>0$.
Then for each $\eps>0$ there is a set $D\subset K^*$ such that
$$\underline d_{(F_N)}(D)\geq\frac\epsilon{\epsilon+\bar d_{(F_N)}(E)-\bar d_{(F_N)}(E)^2}$$
and for all $u\in D$ we have
$$\bar d_{(F_N)}\big((E-u)\cap(E/u)\big)>\bar d_{(F_N)}(E)^2-\eps$$
\end{theorem}

This result is of independent interest and does not follow from Theorem \ref{teo_main}, because here we just need $E$ to satisfy $\bar d_{(F_N)}(E)>0$, not that it is a cell in a finite coloring. Theorem \ref{teo_double} will be proved in Subsection \ref{sec_preliminaries_proofs} as a consequence of Corollary \ref{cor_estimates}, which in turn is proved in Section \ref{section_proof}.

We will need the following lemma, which, roughly speaking, asserts that certain transformations of F\o lner sequences are still F\o lner sequences.
\begin{lemma}\label{lemma_folner}
Let $(F_N)$ be a double F\o lner sequence in a field $K$ and let $b\in K^*$.
Then the sequence $(bF_N)$ is also a double F\o lner sequence.
Also, if $(F_N)$ is a F\o lner sequence for the multiplicative group $(K^*,\times)$, then the sequence $\left(F_N^{-1}\right)$, where $F_N^{-1}=\{g^{-1}:g\in F_N\}$, is still a F\o lner sequence for that group.
\end{lemma}
\begin{proof}
The sequence $(bF_N)$ is trivially a F\o lner sequence for the multiplicative group.
To prove that it is also a F\o lner sequence for the additive group, let $x\in F$, we have
$$\lim_{N\to\infty}\frac{|bF_N\cap(x+bF_N)|}{|bF_N|}=\lim_{N\to\infty}\frac{\left|b\big(F_N\cap(x/b+F_N)\big)\right|}{|F_N|}=1$$
To prove that $\left(F_N^{-1}\right)$ is a F\o lner sequence for the multiplicative group note that for any finite sets $A,B\subset K$ we have $\left|A^{-1}\right|=|A|$, $(A\cap B)^{-1}=A^{-1}\cap B^{-1}$ and if $x\in K^*$ then $(xA)^{-1}=x^{-1}A^{-1}$. Putting all together we conclude that
\begin{eqnarray*}
\lim_{N\to\infty}\frac{\left|F_N^{-1}\cap(xF_N^{-1})\right|}{\left|F_N^{-1}\right|}&=&\lim_{N\to\infty}\frac{\left|\big(F_N\cap(x^{-1}F_N)\big)^{-1}\right|}{|F_N|} \\&=&\lim_{N\to\infty}\frac{\left|F_N\cap(x^{-1}F_N)\right|}{|F_N|}=1
\end{eqnarray*}

\end{proof}

\subsection{A Correspondence Principle}\label{section_correspondence}

To prove Theorem \ref{teo_double} we need an extension of Furstenberg's Correspondence Principle for an action of a group on a set (the classical versions deal with the case when the group acts on itself by translations, cf. \cite{Furstenberg77}).
\begin{theorem}\label{prop_correspondence}
Let $X$ be a set, let $G$ be a countable group and let $(\tau_g)_{g\in G}$ be an action of $G$ on $X$.
Assume that there exists a sequence $(G_N)$ of finite subsets of $X$ such that for each $g\in G$ we have the property:
\begin{equation}\label{eq_folner}\frac{\left|G_N\cap (\tau_gG_N)\right|}{|G_N|}\to1\text{  as  }N\to\infty\end{equation}
Let $E\subset X$ and assume that $\bar d_{(G_N)}(E):=\limsup_{N\to\infty}\frac{|G_N\cap E|}{|G_N|}>0$.

Then there exists a compact metric space $\Omega$, a probability measure $\mu$ on the Borel sets of $\Omega$, a $\mu$-preserving $G$-action $(T_g)_{g\in G}$ on $\Omega$, a Borel set $B\subset\Omega$ such that $\mu(B)=\bar d_{(G_N)}(E)$, and for any $k\in\N$ and $g_1,\dots,g_k\in G$ we have
$$\bar d_{(G_N)}\left(\tau_{g_1}E\cap\dots\cap \tau_{g_k}E\right)\geq\mu\left(T_{g_1}B\cap...\cap T_{g_k}B\right)$$
\end{theorem}

\begin{proof}
Define the family of sets
$${\mathcal S}:=\left\{\bigcap_{j=1}^k\tau_{g_j}E:k\in\N, g_j\in G\ \forall j=1,\dots,k\right\}\cup\{X\}$$
Note that ${\mathcal S}$ is countable, so using a diagonal procedure we can find a subsequence $(\tilde G_N)$ of the sequence $(G_N)$ such that $\displaystyle\bar d_{(G_N)}(E)=\lim_{N\to\infty}|E\cap \tilde G_N|/|\tilde G_N|$ and, for each $S\in{\mathcal S}$, the following limit exists
$$\lim_{N\to\infty}\frac{|S\cap \tilde G_N|}{|\tilde G_N|}$$

Note that (\ref{eq_folner}) holds for any subsequence of $(G_N)$, and in particular for $(\tilde G_N)$.
Let $B(X)$ be the space of all bounded complex-valued functions on $X$.
The space $B(X)$ is a Banach space with respect to the norm $\|f\|=\sup_{x\in X}|f(x)|$.
Let $\rho\in\ell^\infty(\N)^*$ be a Banach limit\footnote{This means that $\rho:\ell^\infty(\N)\to\C$ is a shift invariant positive linear functional such that for any convergent sequence ${\bf x}=(x_n)\in\ell^\infty(\N)$ we have $\rho({\bf x})=\lim x_n$.}.

Define the linear functional $\lambda:B(X)\to \C$ by
$$\lambda(f)=\rho\left(\left(\frac1{|\tilde G_N|}\sum_{x\in \tilde G_N}f(x)\right)_{N\in\N}\right)$$

The functional $\lambda$ is positive (i.e. if $f\geq0$ then $\lambda(f)\geq0$) and $\lambda(1)=1$.
For any $f\in B(X)$, $g\in G$ and $x\in X$, the equation $f_g(x)=f(\tau_gx)$ defines a new function $f_g\in B(X)$.
By (\ref{eq_folner}) we have that $\lambda(f_g)=\lambda(f)$ for all $g\in G$, so $\lambda$ is an invariant mean for the action $(\tau_g)_{g\in G}$.
Moreover, $\bar d_{(G_N)}(E)=\lambda(1_E)$ and, for any $S\in\mathcal S$, we have $\bar d_{(G_N)}(S)\geq\lambda(1_S)$.

Note that the Banach space $B(X)$ is a commutative $C\sp*$-algebra (with the involution being pointwise conjugation).
Now let $Y\subset B(X)$ be the (closed) subalgebra generated by the indicator functions of sets in ${\mathcal S}$.
Then $Y$ is itself a $C^*$-algebra.
It has an identity (the constant function equal to $1$) because $X\in{\mathcal S}$.
If $f\in Y$ then $f_g\in Y$ for all $g\in G$.
Moreover, since ${\mathcal S}$ is countable, $Y$ is separable.
Thus, by the Gelfand representation theorem (cf. \cite{Averson76}, Theorem 1.1.1), there exists a compact metric space $\Omega$ and a map $\Phi:Y\to C(\Omega)$ which is simultaneously an algebra isomorphism and a homeomorphism.

The linear functional $\lambda$ induces a positive linear functional $L$ on $C(\Omega)$ by $L\big(\Phi(f)\big)=\lambda(f)$. Applying the Riesz Representation Theorem we have a measure $\mu$ on the Borel sets of $\Omega$ such that
$$\lambda(f)=L\big(\Phi(f)\big)=\int_\Omega\Phi(f)d\mu\qquad\qquad\forall \,f\in B(X)$$

The action $(\tau_g)_{g\in G}$ induces an anti-action (or right action) $(U_g)_{g\in G}$ on $C(\Omega)$ by $U_g\Phi(f)=\Phi(f_g)$, where $f_g(x)=f(\tau_gx)$ for all $g\in G$, $f\in Y$ and $x\in X$.
It is not hard to see that, for each $g\in G$, $U_g$ is a positive invertible isometry of $C(\Omega)$.
By the Banach-Stone theorem (\cite{Stone37}), for each $g\in G$, there is a homeomorphism $T_g:\Omega\to\Omega$ such that $U_g\phi=\phi\circ T_g$ for all $\phi\in C(\Omega)$.
Moreover for all $g,h\in G$ we have $\phi\circ T_{gh}=U_{gh}\phi=U_hU_g\phi=U_h(\phi\circ T_g)=\phi\circ(T_g\circ T_h)$. This means that $(T_g)_{g\in G}$ is an action  of $G$ on $\Omega$.
For every $f\in Y$ we have $\lambda(f_g)=\lambda(f)$ and hence
\begin{eqnarray*}
\int_\Omega\Phi(f)\circ T_gd\mu&=&\int_\Omega U_g\Phi(f)d\mu=\int_\Omega\Phi(f_g)d\mu\\&=&\lambda(f_g)=\lambda(f)=\int_\Omega\Phi(f)d\mu
\end{eqnarray*}
Therefore the action $(T_g)$ preserves measure $\mu$.

Note that the only idempotents of the algebra $C(\Omega)$ are indicator functions of sets.
Therefore, given any set $S\in\mathcal S$, the Gelfand transform $\Phi(1_S)$ of the characteristic function $1_S$ of $S$ is the characteristic function of some Borel subset (which we denote by $\Phi(S)$) in $\Omega$.
In other words, $\Phi(S)$ is such that $\Phi(1_S)=1_{\Phi(S)}$.
Let $B=\Phi(E)$.
We have
$$\bar d_{(G_N)}(E)=\lambda(1_E)=\int_\Omega\Phi(1_E)d\mu=\int_\Omega1_Bd\mu=\mu(B)$$

Since the indicator function of the intersection of two sets is the product of the indicator functions, we conclude that for any $k\in\N$ and any $g_1,...,g_k\in G$ we have
\begin{eqnarray*}\bar d_{(G_N)}\left(\bigcap_{i=1}^k\tau_{g_i}E\right)&\geq&\lambda\left(\prod_{i=1}^k1_{\tau_{g_i}E}\right)= \int_\Omega\Phi\left(\prod_{i=1}^k1_{\tau_{g_i}E}\right)d\mu\\&=& \int_\Omega\prod_{i=1}^k\Phi\left(1_{\tau_{g_i}E}\right)d\mu= \int_\Omega\prod_{i=1}^kU_{g_i^{-1}}\Phi\left(1_E\right)d\mu\\&=&\int_\Omega\prod_{i=1}^k1_B\circ T_{g_i^{-1}}d\mu= \int_\Omega\prod_{i=1}^k1_{T_{g_i}B}d\mu=\mu\left(\bigcap_{i=1}^kT_{g_i}B\right)\end{eqnarray*}
\end{proof}
\subsection{Deriving Theorem \ref{teo_double} from ergodic results}\label{sec_preliminaries_proofs}
In this subsection we state the main ergodic results of the paper and use them to derive Theorem \ref{teo_double}.
The proof of the ergodic results will be given in Section \ref{section_proof}.
We begin by recalling Theorem \ref{thm_introduction} stated in the introduction.

\begin{theorem}\label{teo_doublergodic}[(cf. Theorem \ref{thm_introduction} in Introduction)]
Let $(\Omega,{\mathcal B},\mu)$ be a probability space and suppose that $\AK$ acts on $\Omega$ by measure preserving transformations. Let $(F_N)$ be a double F\o lner sequence on $K$.
Then for each $B\in{\mathcal B}$ we have\footnote{By slight abuse of language we use the same symbol to denote the elements (such as $M_{1/u}$ and $A_{-u}$) of $\AK$ and the measure preserving transformation they induce on $\Omega$.}
$$\lim_{N\to\infty}\frac1{|F_N|}\sum_{u\in F_N}\mu(A_{-u}B\cap M_{1/u}B)\geq\mu(B)^2$$
and, in particular the limit exists.
\end{theorem}
In the case when the action of $\AK$ is ergodic, we can replace one of the sets $B$ with another set $C$. This is the content of the next theorem.
\begin{theorem}\label{thm_supergodic}
Let $(\Omega,{\mathcal B},\mu)$ be a probability space and suppose that $\AK$ acts ergodically on $\Omega$ by measure preserving transformations.
Let $(F_N)$ be a double F\o lner sequence on $K$.
Then for any $B,C\in{\mathcal B}$ we have
$$\lim_{N\to\infty}\frac1{|F_N|}\sum_{u\in F_N}\mu(A_{-u}B\cap M_{1/u}C)=\mu(B)\mu(C)$$
and, in particular, the limit exists.
\end{theorem}
\begin{remark}
We note that Theorem \ref{thm_supergodic} fails without ergodicity. Indeed, take the normalized disjoint union of two copies of the same measure preserving system. Choosing $B$ to be one of the copies and $C$ the other we get
$$A_{-u}B\cap M_{1/u}C=\emptyset\text{   for all }u\in K^*$$
\end{remark}

We can extract some quantitative bounds from Theorems \ref{teo_doublergodic} and \ref{thm_supergodic}. This is summarized in the next corollary.

\begin{corollary}\label{cor_estimates}
Let $(\Omega,{\mathcal B},\mu)$ be a probability space and suppose that $\AK$ acts on $\Omega$ by measure preserving transformations.
Let $(F_N)$ be a double F\o lner sequence on $K$, let $B\in{\mathcal B}$ and let $\epsilon>0$.
Then we have
$$\underline d_{(F_N)}\left(\left\{u\in K^*:\mu(A_{-u}B\cap M_{1/u}B)>\mu(B)^2-\eps\right\}\right)\geq\frac\epsilon{\epsilon+\mu(B)-\mu(B)^2}$$
Moreover, if the action of $\AK$ is ergodic and $B,C\in{\mathcal B}$, the set
$$D_\eps:=\left\{u\in K^*:\mu(A_{-u}B\cap M_{1/u}C)>\mu(B)\mu(C)-\eps\right\}$$ satisfies
\begin{equation}\label{eq_BC}\underline d_{(F_N)}(D_\eps)\geq\max\left(\frac\epsilon{\epsilon+\mu(B)(1-\mu(C))},\frac\epsilon{\epsilon+\mu(C)(1-\mu(B))}\right)\end{equation}
\end{corollary}

Corollary \ref{cor_estimates} will be proved in Section \ref{section_proof}. We will use it now, together with the correspondence principle, to deduce Theorem \ref{teo_double}.

\begin{proof}[Proof of Theorem \ref{teo_double}]
Let $X=K$, let $G=\AK$ and let $(G_N)=(F_N)$. Applying the correspondence principle (Theorem \ref{prop_correspondence}), we obtain, for each $E\subset K$, a measure preserving action $(T_g)_{g\in\AK}$ of $\AK$ on a probability space $(\Omega,{\mathcal B},\mu)$, a set $B\in{\mathcal B}$ such that $\mu(B)=\bar d_{(F_N)}(E)$, and for all $u\in K^*$ we have $\bar d_{(F_N)}(A_{-u}E\cap M_{1/u}E)\geq\mu(T_{A_{-u}}B\cap T_{M_{1/u}}B)$.
To simplify notation we will denote the measure preserving transformations $T_{A_{-u}}$ and $T_{M_{1/u}}$ on $\Omega$ by just $A_{-u}$ and $M_{1/u}$. Also, recalling that $A_{-u}E=E-u$ and $M_{1/u}E=E/u$ we can rewrite the previous equation as
$$\bar d_{(F_N)}(E-u\cap E/u)\geq\mu(A_{-u}B\cap M_{1/u}B)\qquad\qquad\forall u\in K^*$$

Now assume that $\bar d_{(F_N)}(E)>0$ and let $\epsilon>0$.
Let
$$D_\eps:=\{u\in K^*:\bar d_{(F_N)}\big((E-u)\cap (E/u)\big)>\bar d_{(F_N)}(E)^2-\eps\}$$
By Corollary \ref{cor_estimates} we have
\begin{eqnarray*}
\underline d_{(F_N)}(D_\eps)&\geq&\underline d_{(F_N)}\left(\left\{u\in K^*:\mu(A_{-u}B\cap M_{1/u}B)>\mu(B)^2-\eps\right\}\right)\\&\geq&\frac\epsilon{\epsilon+\mu(B)-\mu(B)^2}\\&=& \frac\epsilon{\epsilon+\bar d_{(F_N)}(E)-\bar d_{(F_N)}(E)^2}
\end{eqnarray*}
\end{proof}

\subsection{Some classical results}
We will need to use two results that are already in the literature.
The first is a version of the classical van der Corput trick for unitary representations of countable abelian groups. For a proof see Lemma 2.9 in \cite{Bergelson_Leibman_McCutcheon05}.
\begin{proposition}\label{lema_vdc}
Let $H$ be an Hilbert space, let $(a_u)_{u\in K^*}$ be a bounded sequence in $H$ indexed by $K^*$. If for all $b$ in a co-finite subset of $K^*$ we have
$$
\limsup_{N\to\infty}\left|\frac1{|F_N|}\sum_{u\in F_N}\langle a_{bu},a_u\rangle\right|=0$$
then also
$$\lim_{N\to\infty}\frac1{|F_N|}\sum_{u\in F_N}a_u=0$$
\end{proposition}
Another result we will need is von Neumann's mean ergodic theorem. See, for instance Theorem 5.5 in \cite{Bergelson06} for a proof of this version.
\begin{theorem}\label{lema_vNmet}
Let $G$ be a countable abelian group and let $(F_N)$ be a F\o lner sequence in $G$.
Let $H$ be a Hilbert space and let $(U_g)_{g\in G}$ be a unitary representation of $G$ on $H$.
Let $P$ be the orthogonal projection onto the subspace of vectors fixed under $G$.
Then
$$\lim_{N\to\infty}\frac1{|F_N|}\sum_{g\in F_N}U_gf=Pf\qquad\qquad\forall f\in H$$
in the strong topology of $H$.
\end{theorem}
\section{Proof of the main theorem}\label{section_proof}
In this section we will prove Theorems \ref{teo_doublergodic} and \ref{thm_supergodic} and Corollary \ref{cor_estimates}.
Throughout this section let $K$ be a countable field, let $(\Omega,{\mathcal B},\mu)$ be a probability space, let $(T_g)_{g\in\AK}$ be a measure preserving action of $\AK$ on $\Omega$ and let $(F_N)$ be a double F\o lner sequence on $K$.

Let $H=L^2(\Omega,\mu)$ and let $(U_g)_{g\in \AK}$ be the unitary Koopman representation of $\AK$ (this means that $(U_gf)(x)=f(g^{-1}x)$).
By a slight abuse of notation we will write $A_uf$ instead of $U_{A_u}f$ and $M_uf$ instead of $U_{M_u}f$.

Let $P_A$ be the orthogonal projection from $H$ onto the subspace of vectors which are fixed under the action of the additive subgroup $S_A$ and let $P_M$ be the orthogonal projection from $H$ onto the subspace of vectors which are fixed under the action of the multiplicative subgroup $S_M$.

We will show that the orthogonal projections $P_A$ and $P_M$ commute, which is surprising considering that the subgroups $S_A$ and $S_M$ do not. The reason for this is that for each $k\in K^*$, the map $M_k:K\to K$ is an isomorphism of the additive group.
\begin{lemma}\label{lema_commute}
For any $f\in H$ we have
$$P_AP_Mf=P_MP_Af$$
\end{lemma}
\begin{proof}We first prove that for any $k\in K^*$, the projection $P_A$ commutes with $M_k$.
For this we will use Theorem \ref{lema_vNmet}, Lemma \ref{lemma_folner} and equation (\ref{eq_affine}):

\begin{eqnarray*}M_kP_Af&=&M_k\left(\lim_{N\to\infty}\frac1{|F_N|}\sum_{u\in F_N}A_uf\right)=\lim_{N\to\infty}\frac1{|F_N|}\sum_{u\in F_N}M_kA_uf\\&=&\lim_{N\to\infty}\frac1{|F_N|}\sum_{u\in F_N}A_{ku}M_kf
=\lim_{N\to\infty}\frac1{|F_N|}\sum_{u\in kF_N}A_uM_kf=P_AM_kf
\end{eqnarray*}
Now we can conclude the result:
$$P_MP_Af=\lim_{N\to\infty}\frac1{|F_N|}\sum_{u\in F_N}M_uP_Af=P_A\left(\lim_{N\to\infty}\frac1{|F_N|}\sum_{u\in F_N}M_uf\right)=P_AP_Mf$$
\end{proof}

Lemma \ref{lema_commute} implies that $P_MP_Af$ is invariant under both $S_A$ and $S_M$. Since those two subgroups generate $\AK$, this means that $P_MP_A$ is the orthogonal projection onto the space of  functions invariant under $\AK$.

Let $P:H\to H$ be the orthogonal projection onto the space of functions invariant under the action of the group $\AK$.
We have $P=P_AP_M=P_MP_A$.

The bulk of the proofs of Theorems \ref{teo_doublergodic} and \ref{thm_supergodic} is the next lemma.

\begin{lemma}\label{lema_bulk}
Let $f\in H=L^2(\Omega,\mu)$.
We have
$$\lim_{N\to\infty}\frac1{|F_N|}\sum_{u\in F_N}M_uA_{-u}f=Pf$$
In particular, the limit exists.
\end{lemma}
\begin{proof}
We assume first that $P_Af=0$.
For $u\in K^*$, let $a_u=M_uA_{-u}f$.
Then for each $b\in K^*$ we have
\begin{eqnarray*}\langle a_{ub},a_u\rangle&=&\langle M_{ub}A_{-ub}f,M_uA_{-u}f\rangle\\ &=&\langle M_bA_{-ub}f,A_{-u}f\rangle\\&=&\langle A_{-ub}f,M_{1/b}A_{-u}f\rangle\\&=&\langle A_{-ub}f,A_{-u/b}M_{1/b}f\rangle\\&=&\langle A_{-u(b-1/b)}f,M_{1/b}f\rangle\end{eqnarray*}
where we used equation (\ref{eq_affine}) and the fact that the operators are unitary.
Now if $b\neq\pm1$ then $b-\frac1b=\frac{b^2-1}b\neq0$ and so the sequence of sets $\left(-\frac {b^2-1}bF_N\right)_N$ is again a double F\o lner sequence on $K$, by Lemma \ref{lemma_folner}.
Thus, applying Theorem \ref{lema_vNmet} we get (keeping $b\neq\pm1$ fixed)
\begin{eqnarray*}\lim_{N\to\infty}\frac1{|F_N|}\sum_{u\in F_N}\langle a_{ub},a_u\rangle&=& \left\langle\lim_{N\to\infty}\frac1{|F_N|}\sum_{u\in F_N}A_{-u(b-1/b)}f,M_{1/b}f\right\rangle\\&=& \left\langle\lim_{N\to\infty}\frac1{|F_N|}\sum_{u\in -\frac {b^2-1}bF_N}A_uf,M_{1/b}f\right\rangle\\&=&\langle P_Af,M_{1/b}f\rangle=0\end{eqnarray*}

Thus it follows from Proposition \ref{lema_vdc} that
$$\lim_{N\to\infty}\frac1{|F_N|}\sum_{u\in F_N}M_uA_{-u}f=0$$

Now, for a general $f\in H$, we can write $f=f_1+f_2$, where $f_1=P_Af$ and $f_2=f-P_Af$ satisfies $P_Af_2=0$.
Note that $f_1$ is invariant under $A_u$. Therefore
\begin{eqnarray*}
\lim_{N\to\infty}\frac1{|F_N|}\sum_{u\in F_N}M_uA_{-u}f&=&\lim_{N\to\infty}\frac1{|F_N|}\sum_{u\in F_N}M_uA_{-u}f_1\\&=&\lim_{N\to\infty}\frac1{|F_N|}\sum_{u\in F_N}M_uf_1\\&=&P_Mf_1=P_MP_Af=Pf
\end{eqnarray*}
\end{proof}
\begin{remark}
Lemma \ref{lema_bulk} can be interpreted as an ergodic theorem along a sparse subset of $\AK$ (namely the subset $\{M_uA_{-u}:u\in K^*\}$).\end{remark}

\begin{proof}[Proof of Theorem \ref{teo_doublergodic}]
Let $B\in{\mathcal B}$.
By Lemma \ref{lema_bulk} applied to the characteristic function $1_B$ of $B$ we get that
\begin{eqnarray*}
\lim_{N\to\infty}\frac1{|F_N|}\sum_{u\in F_N}\mu(A_{-u}B\cap M_{1/u}B)&=&\lim_{N\to\infty}\frac1{|F_N|}\sum_{u\in F_N}\int_\Omega A_{-u}1_B M_{1/u}1_Bd\mu\\&=&\lim_{N\to\infty}\frac1{|F_N|}\sum_{u\in F_N}\int_\Omega (M_uA_{-u}1_B)1_Bd\mu\\&=&\int_\Omega (P1_B)1_Bd\mu
\end{eqnarray*}
We can use the Cauchy-Schwartz inequality with the functions $P1_B$ and the constant function $1$, and the trivial observation that $P1=1$, to get
$$\int_\Omega (P1_B)1_Bd\mu=\|P1_B\|^2\geq\langle P1_B,1\rangle^2=\langle 1_B,P1\rangle^2=\langle1_B,1\rangle^2=\mu(B)^2$$
Putting everything together we obtain
$$\lim_{N\to\infty}\frac1{|F_N|}\sum_{u\in F_N}\mu(A_{-u}B\cap M_{1/u}B)\geq\mu(B)^2$$
\end{proof}

\begin{proof}[Proof of Theorem \ref{thm_supergodic}]
Let $B,C\in{\mathcal B}$.
By Lemma \ref{lema_bulk} applied to the characteristic function $1_B$ of $B$ we get that
\begin{eqnarray*}
\lim_{N\to\infty}\frac1{|F_N|}\sum_{u\in F_N}\mu(A_{-u}B\cap M_{1/u}C)&=&\lim_{N\to\infty}\frac1{|F_N|}\sum_{u\in F_N}\int_\Omega A_{-u}1_B M_{1/u}1_Cd\mu\\&=&\lim_{N\to\infty}\frac1{|F_N|}\sum_{u\in F_N}\int_\Omega (M_uA_{-u}1_B)1_Cd\mu\\&=&\int_\Omega (P1_B)1_Cd\mu
\end{eqnarray*}
Since the action of $\AK$ is ergodic, $P1_B=\mu(B)$, and hence
$$\lim_{N\to\infty}\frac1{|F_N|}\sum_{u\in F_N}\mu(A_{-u}B\cap M_{1/u}C)=\mu(B)\int_\Omega1_Cd\mu=\mu(B)\mu(C)$$

\end{proof}

\begin{proof}[Proof of Corollary \ref{cor_estimates}]
Let $B,C\subset{\mathcal B}$.
Note that trivially $\mu(A_{-u}B\cap M_{1/u}B)\leq\mu(M_{1/u}B)=\mu(B)$.
For each $\epsilon>0$ let $D_\epsilon$ be the set $D_\epsilon:=\{u\in K:\mu(A_{-u}B\cap M_{1/u}B)>\mu(B)^2-\epsilon\}$.

Now let $(\tilde F_N)_{N\in\N}$ be a subsequence of $(F_N)_{N\in\N}$ such that
$$\underline d_{(F_N)}(D_\eps)=\lim_{N\to\infty}\frac{|D_\eps\cap\tilde F_N|}{|\tilde F_N|}$$
Thus $\underline d_{(F_N)}(D_\eps)=\underline d_{(\tilde F_N)}(D_\eps)=\bar d_{(\tilde F_N)}(D_\eps)$.
By Theorem \ref{teo_doublergodic}, we now have
\begin{eqnarray*}
  \mu(B)^2&=&\lim_{N\to\infty}\frac1{|F_N|}\sum_{u\in F_N}\mu(A_{-u}B\cap M_{1/u}B)\\&=&\lim_{N\to\infty}\frac1{|\tilde F_N|}\sum_{u\in\tilde F_N}\mu(A_{-u}B\cap M_{1/u}B)\\&=&\lim_{N\to\infty}\frac1{|\tilde F_N|}\left(\sum_{u\in\tilde F_N\cap D_\eps}\mu(A_{-u}B\cap M_{1/u}B)+\sum_{u\in\tilde F_N\setminus D_\eps}\mu(A_{-u}B\cap M_{1/u}B)\right)\\&\leq&\mu(B)\bar d_{(\tilde F_N)}(D_\eps)+\big(\mu(B)^2-\eps\big)\big(1-\underline d_{(\tilde F_N)}(D_\eps)\big)
\\&=&\mu(B)\underline d_{(F_N)}(D_\eps)+\big(\mu(B)^2-\eps\big)\big(1-\underline d_{(F_N)}(D_\eps)\big)
\end{eqnarray*}
From this we conclude that $\underline d_{(F_N)}(D_\eps)\geq\eps/\Big(\eps+\mu(B)\big(1-\mu(B)\big)\Big)$.

Now assume that the action of $\AK$ is ergodic.
Note that trivially $\mu(A_{-u}B\cap M_{1/u}C)\leq\mu(M_{1/u}C)=\mu(C)$.
For each $\epsilon>0$ let $D_\epsilon$ be the set $D_\epsilon:=\{u\in K:\mu(A_{-u}B\cap M_{1/u}C)>\mu(B)\mu(C)-\epsilon\}$.

Now let $(\tilde F_N)_{N\in\N}$ be a subsequence of $(F_N)_{N\in\N}$ such that
$$\underline d_{(F_N)}(D_\eps)=\lim_{N\to\infty}\frac{|D_\eps\cap\tilde F_N|}{|\tilde F_N|}$$
Thus $\underline d_{(F_N)}(D_\eps)=\underline d_{(\tilde F_N)}(D_\eps)=\bar d_{(\tilde F_N)}(D_\eps)$.
By Theorem \ref{thm_supergodic} we now have
\begin{eqnarray*}
  \mu(B)\mu(C)&=&\lim_{N\to\infty}\frac1{|F_N|}\sum_{u\in F_N}\mu(A_{-u}B\cap M_{1/u}C)\\&=&\lim_{N\to\infty}\frac1{|\tilde F_N|}\sum_{u\in\tilde F_N}\mu(A_{-u}B\cap M_{1/u}C)\\&=&\lim_{N\to\infty}\frac1{|\tilde F_N|}\left(\sum_{u\in\tilde F_N\cap D_\eps}\mu(A_{-u}B\cap M_{1/u}C)+\sum_{u\in\tilde F_N\setminus D_\eps}\mu(A_{-u}B\cap M_{1/u}C)\right)\\&\leq&\mu(C)\bar d_{(\tilde F_N)}(D_\eps)+\big(\mu(B)\mu(C)-\eps\big)\big(1-\underline d_{(\tilde F_N)}(D_\eps)\big)\\&=&\mu(C)\underline d_{(F_N)}(D_\eps)+\big(\mu(B)\mu(C)-\eps\big)\big(1-\underline d_{(F_N)}(D_\eps)\big)
\end{eqnarray*}
From this we conclude that $\underline d_{(F_N)}(D_\eps)\geq\eps/\Big(\eps+\mu(C)\big(1-\mu(B)\big)\Big)$. Switching the roles of $B$ and $C$ we obtain Equation (\ref{eq_BC}).
\end{proof}
\begin{remark}
Note that the lower bound on $\underline d_{(F_N)}(D_\eps)$ does not depend on the set $B$, only on the measure $\mu(B)$.
Moreover, it does not depend on the double F\o lner sequence $(F_N)$.
\end{remark}

\section{Proof of Theorem \ref{teo_main}}\label{section}
In this section we give a proof of Theorem \ref{teo_main}. We start by giving a more precise statement:
\begin{theorem}\label{thm_fullcolor}
For any finite coloring $K=\bigcup C_i$ there exists a color $C_i$, a subset $D\subset K$ satisfying $\bar d_{(F_N)}(D)>0$ and, for each $u\in D$, there is a set $D_u\subset K$ also satisfying $\bar d_{(F_N)}(D_u)>0$ such that for any $v\in D_u$ we have $\{u,u+v,uv\}\subset C_i$.
\end{theorem}

\begin{definition}
Let $G$ be a group.
A set $R\subset G$ is a \emph{set of recurrence} if for all probability preserving actions $(\Omega,\mu,(T_g)_{g\in G})$ and every measurable set $B\subset \Omega$ with positive measure, there exists some non-identity $g\in R$ such that $\mu(B\cap T_gB)>0$.
\end{definition}
The proof of Theorem \ref{thm_fullcolor} uses the fact that sets of recurrence are partition regular.
For other similar applications of this phenomenon see for instance \cite{Bergelson86}, the discussion before Question 11 in \cite{Bergelson96} and Theorem 0.4 in \cite{Bergelson_McCutcheon96}.

The following lemma is well known; we include the proof for the convenience of the reader.
\begin{lemma}\label{lemma_recurrence}
Let $G$ be a group and let $R\subset G$ be a set of recurrence. Then for every finite partition $R=R_1\cup\cdots\cup R_r$, one of the sets $R_i$ is also a set of recurrence.
\end{lemma}
\begin{proof}
The proof goes by contradiction. Assume that none of the sets $R_1,\dots,R_r$ is a set of recurrence. Then for each $i=1,\dots,r$ there is some probability preserving action $(\Omega_i,\mu_i,(T_g)_{g\in G}^{(i)})$ and a set $B_i\subset \Omega_i$ with $\mu_i(B_i)>0$ and such that $\mu_i(B_i\cap T_g^{(i)}B_i)=0$ for all $g\in R_i$.

Let $\Omega=\Omega_1\times\cdots\times\Omega_r$, let $\mu=\mu_1\otimes\cdots\otimes\mu_r$, let $B=B_1\times\dots\times B_r$ and, for each $g\in G$, let $T_g(\omega_1,\dots,\omega_r)=(T_g^{(1)}\omega_1,\dots,T_g^{(r)}\omega_r)$.
Then $(T_g)_{g\in G}$ is a probability preserving action of $G$ on $\Omega$ and $\mu(B)=\mu_1(B_1)\cdots\mu_r(B_r)>0$.

Since $R$ is a set of recurrence, there exists some $g\in R$ such that $\mu(B\cap T_gB)>0$.
Since $\mu(B\cap T_gB)=\prod_{i=1}^r\mu_i(B_i\cap T_gB_i)$ we conclude that $\mu_i(B_i\cap T_gB_i)>0$ for all $i=1,\dots,r$. But this implies that $g\notin R_i$ for all $i=1,\dots,r$, which contradicts the fact that $g\in R=R_1\cup\dots\cup R_r$.
\end{proof}
\begin{proof}[Proof of Theorem \ref{thm_fullcolor}]

Let $K=C_1\cup C_2\cup...\cup C_{r'}$ be a finite partition of $K$.
Assume without loss of generality that, for some $r\leq r'$, the upper density $\bar d_{(F_N)}(C_i)$ is positive for $i=1,...,r$ and $\bar d_{(F_N)}(C_i)=0$ for $i=r+1,...,r'$.

For a set $C\subset K$ and each $u\in C$ define the set $D_u(C)=(C-u)\cap (C/u)$. Let $D(C)=\left\{u\in C:\bar d_{(F_N)}\big(D_u(C)\big)>0\right\}$. We want to show that for some $i=1,...,r$ we have $\bar d_{(F_N)}\big(D(C_i)\big)>0$.

If for some $1\leq i\leq r$ we have $\bar d_{(F_N)}\big(D(C_i)\big)=0$ but $D(C_i)\neq\emptyset$, we can consider the more refined coloring obtained by distinguish $D(C_i)$ and $C_i\setminus D(C_i)$.
Since
$$D\big(C_i\setminus D(C_i)\big)\subset \big(C_i\setminus D(C_i)\big)\cap D(C_i)$$
we conclude that $D\big(C_i\setminus D(C_i)\big)=\emptyset$.
Thus, without loss of generality, we can assume that either $D(C_i)=\emptyset$ or $\bar d_{(F_N)}\big(D(C_i)\big)>0$. Therefore it suffices to show that for some $i=1,...,r$ we have $D(C_i)\neq\emptyset$.

For each $i=1,...,r$ let $R_i=\{M_uA_{-u}:u\in C_i\}\subset\AK$ and let $R=R_1\cup...\cup R_r$.
We claim that $R$ is a set of recurrence.
Indeed, given any probability preserving action $(\Omega,\mu,(T_g)_{g\in\AK})$ of $\AK$ and any measurable set $B\subset\Omega$ with positive measure, by Theorem \ref{teo_doublergodic} we find that the set $\{u\in K^*:\mu(A_{-u}B\cap M_{1/u}B)>0\}$ has positive upper density. In particular, for some $u\in C_1\cup...\cup C_r$ we have that $\mu(M_uA_{-u}B\cap B)=\mu(A_{-u}B\cap M_{1/u}B)>0$.
Since $M_uA_{-u}\in R$ we conclude that $R$ is a set of recurrence.

By Lemma \ref{lemma_recurrence} we conclude that for some $i=1,...,r$ the set $R_i$ is a set of recurrence.
We claim that $D(C_i)\neq\emptyset$.

To see this, apply the correspondence principle (Theorem \ref{prop_correspondence}) with $X=K$, $G=\AK$, $G_N=F_N$ and $E=C_i$ to find a probability preserving action $(T_g)_{g\in\AK}$ of $\AK$ on some probability space $(\Omega,\mu)$ and a measurable set $B\subset \Omega$ satisfying $\mu(B)=\bar d_{(F_N)}(C_i)$ and
$$\bar d_{(F_N)}\left(A_{-u}C_i\cap M_{1/u}C_i\right)\geq\mu\left(T_{A_{-u}}B\cap T_{M_{1/u}}B\right)$$
for all $u\in K^*$.
Since $R_i$ is a set of recurrence, there is some $u\in C_i$ such that
\begin{eqnarray*}
0&<&\mu(T_{M_uA_{-u}}B\cap B)=\mu\left(T_{A_{-u}}B\cap T_{M_{1/u}}B\right)\\&\leq&\bar d_{(F_N)}\left(A_{-u}C_i\cap M_{1/u}C_i\right)=\bar d_{(F_N)}\big(D_u(C_i)\big)
\end{eqnarray*}
We conclude that $u\in D(C_i)$, hence $\bar d_{(F_N)}(D^i)>0$.

Let $D=D(C_i)\subset C_i$ and for each $u\in D$ let $D_u=D_u(C_i)$.
Now let $v\in D_u$.
Then we have $u+v\in C_i$ and $uv\in C_i$.
We conclude that $\{u,u+v,uv\}\subset C_i$ as desired.
\end{proof}

\section{Finite Fields}\label{section_finite}
The main result of this section is an analog of Theorem \ref{teo_doublergodic} for finite fields.

For a finite field $F$, let $F^*=F\setminus\{0\}$ be the multiplicative subgroup.
The group of affine transformations, which we denote by $\AF$, is the group of maps of the form $x\mapsto ux+v$ where $u\in F^*$ and $v\in F$.
Again we will use the notation $A_u\in \AF$ to denote the map $x\mapsto x+u$ and $M_u\in \AF$ to denote the map $x\mapsto ux$, and we will use the subgroups $S_A$ and $S_M$ as defined in Definition \ref{def_subgroups}.
The next result is an analogue of Theorem \ref{teo_doublergodic} for finite fields:
\begin{theorem}\label{teo_finitergodic}
Let $F$ be a finite field and assume that the affine group $\AF$ acts by measure preserving transformations on a probability space $(\Omega,{\mathcal B},\mu)$.
Then for each $B\in{\mathcal B}$ such that $\mu(B)>\sqrt{6/|F|}$ there exists $u\in F^*$ such that $\mu(B\cap M_uA_{-u}B)>0$.

Moreover, if the action of $\AF$ on $\Omega$ is ergodic (this is the case, for instance, when $\Omega=F$) and if $B,C\in{\mathcal B}$ are such that $\mu(B)\mu(C)>6/|F^*|$, then there exists $u\in F^*$ such that $\mu(B\cap M_uA_{-u}C)>0$.
\end{theorem}

For an estimation on how many $u\in F^*$ satisfy Theorem \ref{teo_finitergodic}, see Corollary \ref{cor_finite} below.

The proof of Theorem \ref{teo_finitergodic}, is a ``finitization'' of the proof of Theorem \ref{teo_doublergodic}.
\begin{proof}
Let $H=L^2(\Omega,\mu)$.
We consider the Koopman representation $(U_g)_{g\in\AF}$ of $\AF$ on $H$ by defining $(U_gf)(x)=f(g^{-1}x)$
By an abuse of notation we will denote $U_{A_u}f$ by just $A_uf$ and $U_{M_u}f$ by just $M_uf$.
Let $P_A$ be the orthogonal projection onto the space of all functions invariant under the additive subgroup $S_A$, so that $\displaystyle P_Af(x)=\frac1{|F|}\sum_{u\in F}A_uf$, and let $P_M$ be the orthogonal projection onto the space of all functions invariant under the multiplicative subgroup $S_M$, so that $\displaystyle P_Mf(x)=\frac1{|F^*|}\sum_{u\in F^*}M_uf$.
We claim that $P_MP_Af=P_AP_Mf$, in analogy with Lemma \ref{lema_commute}.
Indeed, by Equation \ref{eq_affine} we have
$$P_MP_Af=\frac1{|F^*|}\sum_{u\in F^*}\frac1{|F|}\sum_{v\in F}M_uA_vf=\frac1{|F||F^*|}\sum_{u\in F^*}\sum_{v\in F}A_{uv}M_uf$$
Since, for each $u\in F^*$, we have $\{uv:v\in F\}=F$, we conclude:
$$P_MP_Af=\frac1{|F||F^*|}\sum_{u\in F^*}\sum_{v\in F}A_vM_uf=P_AP_Mf$$
proving the claim.
Let $B,C\in{\mathcal B}$ be such that $\mu(B)\mu(C)>6/|F^*|$ and let $f=1_C-P_A1_C$.
Note that $P_Af=0$ and $A_uP_A1_C=P_A1_C$.
Since we are in a finite setting now, we will need to bound error terms that are not $0$ (but asymptotically go to $0$ as $|F|$ increases to $\infty$).
For that we will need an estimation on the norm of $f$.

We have, for every $u\in F^*$, that $|A_u1_C-1_C|$ is the indicator function of a set with measure no larger than $2\mu(C)$. Hence $\left\|1_C-A_u1_C\right\|\leq \sqrt{2\mu(C)}$. Therefore
$$\|f\|=\left\|\frac1{|F|}\sum_{u\in F}1_C-A_u1_C\right\|\leq\frac1{|F|}\sum_{u\in F}\left\|1_C-A_u1_C\right\|\leq \sqrt{2\mu(C)}$$

We need to estimate the sum of the measures of the intersections $B\cap~ M_uA_{-u}C$ with $u$ running over all possible values in $F^*$:
\begin{equation}\label{eq_finite1}\sum_{u\in F^*}\mu(B\cap M_uA_{-u}C)=\sum_{u\in F^*}\langle1_B,M_uA_{-u}P_A1_C\rangle+\sum_{u\in F^*}\langle1_B,M_uA_{-u}f\rangle\end{equation}
By the Cauchy-Schwartz inequality we have that
\begin{equation}\label{eq_finite2}\left|\sum_{u\in F^*}\langle1_B,M_uA_{-u}f\rangle\right|=\left|\left\langle1_B,\sum_{u\in F^*}M_uA_{-u}f\right\rangle\right|\leq\sqrt{\mu(B)}\left\|\sum_{u\in F^*}M_uA_{-u}f\right\|\end{equation}
Using linearity of the inner product, the fact that the operators $A_u$ and $M_u$ are unitary, equation (\ref{eq_affine}) and the fact that $F^*$ is a multiplicative group (so that we can change the variables in the sums while still adding over the whole group) we get
\begin{eqnarray*}
\left\|\sum_{u\in F^*}M_uA_{-u}f\right\|^2&=&\sum_{u,d\in F^*}\langle M_uA_{-u}f,M_dA_{-d}f\rangle\\&=&\sum_{u,d\in F^*}\langle A_{d^2/u-u}f,M_{d/u}f\rangle\\&=&\sum_{u,d\in F^*}\langle A_{u(d^2-1)}f,M_df\rangle
\end{eqnarray*}
Now we separate the sum when $d=\pm1$ and note that when $d\neq\pm1$ we have
$$\sum_{u\in F^*}\langle A_{u(d^2-1)}f,M_df\rangle=\langle |F|P_Af-f,M_df\rangle=-\langle f,M_df\rangle$$
so putting this together we obtain:
$$\left\|\sum_{u\in F^*}M_uA_{-u}f\right\|^2=|F^*|(\langle f,M_1f+M_{-1}f\rangle)-\sum_{d\neq\pm1}\langle f,M_d f\rangle$$
Applying again the Cauchy-Schwartz inequality and using the bound $\|f\|\leq\sqrt{2\mu(C)}$, we get the estimate
\begin{equation}\label{eq_finite3}\left\|\sum_{u\in F^*}M_uA_{-u}f\right\|^2\leq3|F^*|\|f\|^2\leq6|F^*|\mu(C)\end{equation}

Combining this with (\ref{eq_finite1}) and (\ref{eq_finite2}) we have
$$\sum_{u\in F^*}\mu(B\cap M_uA_{-u}C)\geq\sum_{u\in F^*}\left\langle1_B,M_uP_A1_C\right\rangle-\sqrt{6|F^*|\mu(B)\mu(C)}$$
Normalizing we conclude that
\begin{equation}\label{eq_finitergodic}
\frac1{|F^*|}\sum_{u\in F^*}\mu(B\cap M_uA_{-u}C)\geq\langle 1_B,P_MP_A1_C\rangle-\sqrt{\frac{6\mu(B)\mu(C)}{|F^*|}}
\end{equation}
Note that $P_MP_A1_C=P_AP_M1_C$ is a function invariant under $\AF$.
Thus, if the action of $\AF$ is ergodic then $P_MP_A1_C=\mu(C)$ .
Therefore the right hand side of the previous inequality is $\mu(B)\mu(C)-\sqrt{6\mu(B)\mu(C)/|F^*|}$, so when  $\mu(B)\mu(C)>6/|F^*|$ it is positive and hence for some $n\in F^*$ we have $\mu(B\cap M_uA_{-u}C)>0$.

When $C=B$, and without assuming ergodicity, we have that $P_MP_A1_B=P_AP_M1_B$ is the projection of $1_B$ onto the subspace of invariant functions under the action of $\AF$.
Therefore
\begin{equation}\label{eq_finiteaverage}
\langle 1_B,P_MP_A1_B\rangle=\|P_MP_A1_B\|^2\geq\langle P_MP_A1_B,1\rangle^2=\langle 1_B,1\rangle^2=\mu(B)^2
\end{equation}
So if $\mu(B)>\sqrt{6/|F^*|}$, the average above is positive and hence $\mu(B\cap M_uA_{-u}B)>0$ for some $u\in F^*$.
\end{proof}

As a Corollary of the proof we get the following estimates:
\begin{corollary}\label{cor_finite}
Let $F$ be a finite field and assume that the affine group $\AF$ acts by measure preserving transformations on a probability space $(\Omega,{\mathcal B},\mu)$.
Then for each $B\in{\mathcal B}$ and for each $\delta<\mu(B)$, the set $D:=\{u\in F^*:\mu(B\cap M_uA_{-u}B)>\delta\}$ satisfies
$$\frac{|D|}{|F^*|}\geq\frac{\mu(B)^2-\mu(B)\sqrt{6/|F^*|}-\delta}{\mu(B)-\delta}$$
Moreover, if the action of $\AF$ on $\Omega$ is ergodic, then for all $B,C\in{\mathcal B}$ and for each $\delta<\min\big\{\mu(B),\mu(C)\big\}$, the set $D:=\{u\in F^*:\mu(B\cap M_uA_{-u}C)>\delta\}$ satisfies
$$\frac{|D|}{|F^*|}\geq\frac{\mu(B)\mu(C)-\sqrt{6\mu(B)\mu(C)/|F^*|}-\delta}{\min\big(\mu(B),\mu(C)\big)-\delta}$$
\end{corollary}

\begin{proof}
Let $B\in{\mathcal B}$ and let $\delta<\mu(B)$.
Let $D:=\{u\in F^*:\mu(B\cap M_uA_{-u}B)>\delta\}$.
From Equations (\ref{eq_finitergodic}) and (\ref{eq_finiteaverage}) we have
$$\frac1{|F^*|}\sum_{u\in F^*}\mu(B\cap M_uA_{-u}B)\geq\mu(B)^2-\mu(B)\sqrt{\frac6{|F^*|}}$$
On the other hand, since $\mu(B\cap M_uA_{-u}B)\leq\mu(B)$ we have
$$\frac1{|F^*|}\sum_{u\in F^*}\mu(B\cap M_uA_{-u}B)\leq\frac{|D|}{|F^*|}\mu(B)+\left(1-\frac{|D|}{|F^*|}\right)\delta=\delta+\frac{|D|}{|F^*|}\big(\mu(B)-\delta\big)$$
Putting both together we obtain the conclusion of Corollary \ref{cor_finite}.
The case when the action is ergodic follows similarly, using Equation (\ref{eq_finitergodic}) and the fact that $P1_C=\mu(C)$ is a constant function.
\end{proof}

An application of Theorem \ref{teo_finitergodic} is the following finitistic analogue of Theorem \ref{teo_double}.

\begin{theorem}\label{teo_finitefield}
For any finite field $F$ and any subsets $E_1,E_2\subset F$ with $\left|E_1\right|\left|E_2\right|>6|F|$, there exist $u,v\in F$, $v\neq0$, such that $u+v\in E_1$ and $uv\in E_2$.

More precisely, for each $s<\min(|E_1|,|E_2|)$ there is a set $D\subset F^*$ with cardinality
$$|D|\geq\frac{\left|E_1\right|\left|E_2\right||F^*|/|F|-\sqrt{6\left|E_1\right|\left|E_2\right|\left|F^*\right|}-s|F^*|}{\min(|E_1|,|E_2|)-s}$$
such that for each $u\in D$ there are $s$ choices of $v\in F$ such that $u+v\in E_1$ and $uv\in E_2$.
\end{theorem}
Since the action of $\AF$ on $F$ is always ergodic we get a slightly stronger result than Theorem \ref{teo_double}, in that we have two sets $E_1$ and $E_2$.
Unfortunately we were unable to apply the methods of Section \ref{section} used to derive Theorem \ref{teo_main} from Theorem \ref{teo_doublergodic} in the finitistic situation.

\begin{proof}[Proof of Theorem \ref{teo_finitefield}]
Let $\Omega=F$, let $\mu$ be the normalized counting measure on $F$ and let $\AF$ act on $F$ by affine transformations.
Note that this action is ergodic.
Let $\delta=s/|F|$ and let $D=\{u\in F^*:\mu(E_2\cap M_uA_{-u}E_1)>\delta\}$.
By Corollary \ref{cor_finite} we have that
\begin{eqnarray*}
\frac{|D|}{|F^*|}&\geq&\frac{\mu(E_1)\mu(E_2)-\sqrt{6\mu(E_1)\mu(E_2)/|F^*|}-\delta}{\min\big(\mu(E_1),\mu(E_2)\big)-\delta}\\&=&\frac{\left|E_1\right|\left|E_2\right|/|F|-\sqrt{6\left|E_1\right|\left|E_2\right|/|F^*|}-s}{\min(|E_1|,|E_2|)-s}
\end{eqnarray*}
and hence
$$|D|\geq\frac{\left|E_1\right|\left|E_2\right||F^*|/|F|-\sqrt{6\left|E_1\right|\left|E_2\right|\left|F^*\right|}-s|F^*|}{\min(|E_1|,|E_2|)-s}$$
For each $u\in D$ we have
$$\frac s{|F|}=\delta\leq\mu(E_2\cap M_uA_{-u}E_1)=\mu(M_{1/u}E_2\cap A_{-u}E_1)=\frac{\left|M_{1/u}E_2\cap A_{-u}E_1\right|}{|F|}$$
Thus we have $s$ choices for $v$ inside $M_{1/u}E_2\cap A_{-u}E_1$ and for each such $v$ we have both $uv\in E_2$ and $u+v\in E_1$.
\end{proof}

Theorem \ref{teo_finitergodic} implies also the following combinatorial result in finite dimensional vector spaces over finite fields.
\begin{corollary}
Let $d\in\N$ and let $F$ be a finite field. Then for each set $B\subset F^d$ with $|B|>\sqrt{6}|F|^{d-1/2}$ and any $\alpha=(\alpha_1,...,\alpha_d)\in (F^*)^d$ there exists $v=(v_1,...,v_d)\in F^d$ and $u\in F^*$ such that both $v+u\alpha:=(v_1+u\alpha_1,...,v_d+u\alpha_d)$ and $vu:=(v_1u,...,v_du)$ are in $B$.
\end{corollary}

\begin{proof}
Let $\Omega=F^d$ and let $\mu$ be the normalized counting measure on $\Omega$.
Note that $\mu(B)>\sqrt{6/|F|}$.
Consider the action of the affine group $\AF$ on $\Omega$ defined coordinate-wise.

By Theorem \ref{teo_finitergodic} we obtain $u\in F^*$ such that $\mu(B\cap M_uA_{-u}B)=\mu(M_{1/u}B\cap A_{-u}B)>0$.
Let $v\in M_{1/u}B\cap A_{-u}B$.
We conclude that both $uv\in B$ and $u+v\alpha\in B$.
\end{proof}

Theorem \ref{teo_finitefield} was obtained by different methods by Cilleruelo \cite{Cilleruelo12} and by Hanson \cite{Hanson13}.
It should also be mentioned that, for fields of prime order, Shkredov obtained a stronger result:
\begin{theorem}[\cite{Shkredov10}]
  Let $F$ be a finite field of prime order and let $A,B,C\subset F$ be such that $\left|A\right|\left|B\right|\left|C\right|>40|F|^{5/2}$. Then there are $x,y\in F$ such that $x+y\in A$, $xy\in B$ and $x\in C$.
\end{theorem}
\section{Some concluding remarks}\label{section_remarks}
\subsection{}
Iterating Theorem \ref{teo_double} one can obtain more complex configurations.
For instance, if $E\subset K^*$ is such that $\bar d_{(F_N)}(E)>0$, then there exist $x,y\in K^*$ such that
\begin{eqnarray*}\bar d_{(F_N)}\left(\Big(\big((E-x)\cap(E/x)\big)-y\Big)\cap\Big(\big((E-x)\cap(E/x)\big)/y\Big)\right)&=&\\
\bar d_{(F_N)}\left((E-x-y)\cap(E/x-y)\cap\big((E-x)/y\big)\cap\big(E/(xy)\big)\right)&>&0
\end{eqnarray*}
In particular there exist $x,y,z\in K^*$ such that $\{z+y+x,(z+y)x,zy+x,zyx\}\subset E$.
Iterating once more we get $x,y,z,t\in K^*$ such that
$$\left\{\begin{array}{lccr}((t+z)+y)+x&((t+z)+y)\times x&((t+z)\times y)+x&((t+z)\times y)\times x\\((t\times z)+y)+x&((t\times z)+y)\times x&((t\times z)\times y)+x&((t\times z)\times y)\times x
\end{array}\right\}\subset E$$

More generally, for each $k\in\N$, applying $k$ times Theorem \ref{teo_double} we find, for a given set $E\subset K^*$ with $\bar d_{(F_N)}(E)>0$, a finite sequence $x_0,x_2,...,x_k$ such that
$$(\dots(((x_0\circ_1x_1)\circ_2x_2)\circ_3x_3)\dots)\circ_kx_k\in E$$
for each of the $2^k$ possible choices of operations $\circ_i\in\{+,\times\}$.
Note that the sequence $x_0,...,x_k$ depends on $k$, so we do not necessarily have an infinite sequence $x_0,x_1,...$ which works for every $k$ (in the same way that we have arbitrarily long arithmetic progressions on a set of positive density but not an infinite arithmetic progression).

\subsection{}

While the main motivation for this paper was Question \ref{problem}, our methods do not work in $\N$, at least without some new ideas.
The crucial difference between the field setup and that of $\N$ (or $\Z$) is that the affine group $\AK$ of a field $K$ is amenable whereas the semigroup $\{ax+b:a,b\in\Z,a\neq0\}$ is not.
In particular, it is not difficult to see that no double F\o lner sequence can exist for $\N$ (or $\Z$).
Indeed, the set $2\N$ of even numbers must have density $1$ with respect to any multiplicative F\o lner sequence because it is a (multiplicative) shift of $\N$.
On the other hand, $2\N$ must have density $1/2$ with respect to any additive F\o lner sequence, because $\N$ is the disjoint union of two (additive) shifts of $2\N$.

Even if for a ring $R$ there exists a double F\o lner sequence, we are not guaranteed to have Lemma \ref{lemma_folner}, which is used to prove Lemma \ref{lema_commute}.
Another interesting question is whether Lemma \ref{lema_commute} holds for measure preserving actions of the semigroup of affine transformations of $\N$.

\subsection{}

Note that the stipulation about arbitrarily `large' in Theorem \ref{teo_main} is essential since we want to avoid the case when the configuration $\{x+y,xy\}$ degenerates to a singleton.
To better explain this point, let $x\in K$, $x\neq1$ and let $y=\frac x{x-1}$.
Then $xy=x+y$ and hence the configuration $\{x+y,xy\}$ is rather trivial.
We just showed that for any finite coloring of $K$ there are infinitely many (trivial) monochromatic configurations of the form $\{x+y,xy\}$.
Note that our Theorem \ref{thm_fullcolor} is much stronger than this statement, not only because we have configurations with $3$ terms $\{x,x+y,xy\}$, but also because for each of "many" $x$ (indeed a set of positive lower density with respect to any double F\o lner sequence) there is not only one but "many" $y$ (indeed a set of positive upper density with respect to any double F\o lner sequence) such that $\{x,x+y,xy\}$ is monochromatic.
\subsection{}

Our main ergodic result (Theorem \ref{teo_doublergodic}) raises the question of whether, under the same assumptions, one has a triple intersection of positive measure $\mu(B\cap A_{-u}B\cap M_{1/u}B)>0$ for some $u\in K^*$.
This would imply that, given any set $E\subset K$ with $\bar d_{(F_N)}(E)>0$, one can find $u,y\in K^*$ such that $\{y,y+u,yu\}\subset E$.
Using the methods of Section \ref{section}, one could then show that for every finite coloring of $K$, one color contains a configuration of the form $\{u,y,y+u,yu\}$.

On the other hand, not every set $E\subset K$ with $\bar d_{(F_N)}(E)>0$ contains a configuration $\{u,y,y+u,yu\}$.
In fact, in every abelian group there exists a syndetic set (hence of positive density for any F\o lner sequence) not containing a configuration of the form $\{u,y,y+u\}$.
Indeed, let $G$ be an abelian group and let $\chi:G\to\R/\Z$ be a non-principal character (a non-zero homomorphism; it exists by Pontryagin duality). Then the set $E:=\{g\in G:\chi(g)\in[1/3,2/3)\}$ has no triple $\{u,y,y+u\}$. However it is syndetic because the intersection $[1/3,2/3)\cap\chi(G)$ is syndetic in the group $\chi(G)$. (This is true and easy to check with $\chi(G)$ replaced by any subgroup of $\R/\Z$.)

\paragraph{\textbf{Acknowledgements}}
The authors thank Donald Robertson for helpful comments regarding an earlier draft of this paper.

\bibliography{refs-joel}
\bibliographystyle{plain}

\end{document}